\documentclass{amsart}
\usepackage{amsfonts,amsmath}

\newtheorem{Lemma}{Lemma}[section]
\newtheorem{Theorem}[Lemma]{Theorem}
\newcommand{\C}{\mathbb{C}}

\begin{document}
\title{On a Conjecture of Harvey and Lawson}

\author{John Wermer}
\maketitle

\def\bbp{{\Bbb P}}

\section{Introduction} 

Let $\gamma$ be a smooth simple closed curve in complex projective space $\bbp^{n}$.

\smallskip
\noindent
{\bf Question:}  Under what conditions on $\gamma$ does there exist a 1-complex dimensional analytic variety  $V$ in $\bbp^{n}$ such that
$\gamma$ is the boundary of $V$ ?
\smallskip

Dolbeault and Henkin in \cite{D-H}, and Harvey and Lawson in \cite{H-L{1}} have studied this problem. Harvey and Lawson introduced the
notion of the  { \bf projective hull}   $\hat{K}$  of a compact set $K$ in $\bbp^{n}$,
which is defined as follows: Fix a point $x$ in $\bbp^{n}$  with homogeneous coordinates 
$[Z] = [Z_{0},...,Z_{n}]$ Let $P$ be a
homogeneous polynomial on $\C^{n+1}$ of degree $d$.  Define
$$
  ||P(x)|| =  \frac{|P(Z)|}{||Z||^{d}}, \ \ {\rm  where} \ \   ||Z||^{2} =        \sum|Z_{j}|^{2}.
$$
Fix a compact set $K$  in $\C^{n}$.  We define the set $\hat{K}$ as the 
collection of points $x$ in $\bbp^{n}$ such that there exists
a constant $C_{x}$ such that

\begin{equation}
||P(x)||  \  \leq   \ C_{x}^{d} \cdot \sup_K ||P||
\end{equation}
for each homogeneous polynomial  $P$ on  $\C^{n+1}$ of degree $d$ , and for all $d$.

It follows from this definition that if $x$ is a point in $\C^{n}$,then $x \epsilon \hat{K}$ if and only if there exists a constant
$c$ such that

\begin{equation}
|p(x)| \leq c^{d} \cdot \sup_K |p|
\end{equation}
for every polynomial  $p$ in $C[z_{1},...,z_{n}]$ of degree $\leq d$. 

Harvey and Lawson made the following Conjecture:

``If $\gamma$ is a real-analytic closed curve in $\C^{n}$,and if $\hat{\gamma} \ne {\gamma} $ then $\hat{\gamma}\setminus \gamma$ is a 1-complex dimensional analytic subvariety of $\bbp^{n} \setminus \gamma$.'' 

If this holds, then $\hat{\gamma}$ is either an algebraic curve which contains $\gamma$
or  $\hat{\gamma}$ is a variety having $\gamma$ as its boundary.

The motivation for requiring real-analyticity of $\gamma$, rather than merely smoothness, is given in \cite{H-L{1}}.

Let next  $X$ be a complex manifold, and denote by $H(X)$ the space of all holomorphic functions on $X$.  Let  $K$  be a compact
subset of $X$. The   { \bf hull of $K$ in $X$}, denoted $h_{X}(K)$, is defined as the set of points $x$  in  $X$ such that

\begin{equation}  
 |F(x)|   \leq \sup_K |F|\quad {\rm for \  all \ }   F \epsilon H(X)
\end{equation}

\begin{Theorem}\label{t1.1}
 Let $\gamma$ be a smooth closed curve in  $\C^{n}$. Assume
 \medskip

(i) \ $\hat{\gamma}$ is closed in $\bbp^{n}$, and
\medskip
 
(ii)  $\Omega$ is a Stein domain in $\bbp^{n}$ with $\hat{\gamma}$ contained in $\Omega$.
\medskip

\noindent
Then $\hat{\gamma} = h_{\Omega}(\gamma)$

\end{Theorem}

\section{ Proof of Theorem 1.1}
\begin{proof}

By hypothesis, there exists a Stein domain $\Omega$ in $\bbp^{n}$ with 
$\hat{\gamma}$ contained in $\Omega$. Also, $\hat{\gamma}$
is closed by hypothesis, and hence compact.

We now define $A$ as the uniform closure on $\hat{\gamma}$ of
 $H(\Omega)$, restricted to $\hat{\gamma}$. Since $\Omega$ is Stein, 
 $H(\Omega)$ separates points of $\hat{\gamma}$, and so $A$ is a 
 uniform algebra on $\hat{\gamma}$ 

Let $y_{0}$ be a peak-point of $A$ on $\hat{\gamma}$,
 i.e. $y_{0}$ is a point of $\hat{\gamma}$ such that there exists $F^{*}$ in $A$,
with $F^{*}(y_{0}) = 1$ and $|F^{*}| < 1$ on $\hat{\gamma} \setminus {y_{0}}$.
 
We claim that $y_{0}$ is in $\gamma$. Suppose not. Then we can 
choose an open neighborhood $U$ of $y_{0}$ in $\bbp^{n}$ with $\bar{U}$
compact and $\bar{U} \cap \gamma = \emptyset$.Without loss of generality, 
$\bar{U}$ is contained in an affine subspce $W$ of $\bbp^{n}$
and $\bar{U}$ is polynomially convex in $W$.

Theorem 12.8 in \cite{H-L{1}}  now yields
 \begin{equation}
  \hat{\gamma} \cap U \ \ {\rm is \  contained \  in\  the \ polynomial \  hull \  of \ } 
\hat{\gamma} \cap \delta U
\end{equation}
It follows that if $P$ is a polynomial on $W$,then
$$
 |P(y_{0})|  \ \leq \ \max|P|\ \ {\rm over }\  \hat{\gamma} \cap\delta U.  
$$
Since $\bar{U}$ is polynomially convex in $W$, every $F$ in $H(\Omega)$
is uniformly approximable on $\bar{U}$ by polynomials on  $W$.
So for $F$ in $H(\Omega)$, we have
 \begin{equation}
    |F(y_{0})| \ \leq \ \max|F| \ \ {\rm over }\  \hat{\gamma} \cap \delta U.
 \end{equation}
Our function $F^{*}$ above satisfies $F^{*}(y_{0}) = 1 $ and 
$|F^{*}| < 1$ on $\hat{\gamma} \cap \delta U$. We choose $F$ in $H(\Omega)$
so close to $F^{*}$ on $\hat{\gamma}$ that
\begin{equation}
   |F(y_{0})| \ > \ \max|F|   \ \ {\rm over }\    \hat{\gamma} \cap \delta U.
\end{equation}
Assertions  (5) and (6) are in contradiction. So $y_{0}$ is in $\gamma$, as claimed.

Choose now an element $F$  in $A$. By Th. 12.10 in \cite{Ga} , there exists a peak-point $p$ of $A$ such that 
$\max|F|$ over $\hat{\gamma}$ equals $|F(p)|$.

By the preceding, $p$ is in $\gamma$. Hence,  $\max|F|$ over $\hat{\gamma} \leq \max|F|$ over $\gamma$. This holds in particular for $F$
in $H(\Omega)$. So we have
 \begin{equation}
     \hat{\gamma} \subset{h_{\Omega}}(\gamma).
 \end{equation}

To prove Theorem 1.1 we need to prove the reverse inclusion. Fix a point $x$ in $h_{\Omega}(\gamma)$ We choose a complex hyperplane  $l$
of $\bbp^{n}$ such that $x$ is not in $l$. Let $z_{1},...,z_{n}$ be affine coordinates on the affine space $\bbp^{n} \setminus {l}$.
Each $z_{j}$ extends as a meromorphic function to $\bbp^{n}$, with pole set $l$.
 
Since $\Omega$ is a Stein manifold, there exists a holomorphic function $\Lambda$ on $\Omega$ such that $\Lambda$ vanishes on $l \cap \Omega$
and $\Lambda(x) \neq 0$. It follows that , for all $j$, $\Lambda \times z_{j}$ is holomorphic on $l \cap\Omega$, and hence is holomorphic on all of $\Omega$.

Let $J$ denote the multi-index $(j_{1},...,j_{n})$, and let $z^{J}$ denote 
the product of the monomials $z_{r}^{j_{r} }$ for $r = 1,...,n$
Let $P$ be the polynomial which is the sum of terms $c_{J}z^{J}$ taken over the multi-indices $J$, where $c_{J}$ is a scalar.  Let $d = deg P$.
Then if $c_J\neq0$, we have 
   $\sum_{s=1}^{n} j_{s} \leq d$.  Hence
$$
\Lambda^{d} \times P = \sum c_{J}\Lambda^{d}z^{J} = \sum c_{J}(\Lambda z_{1})^{j_{1}}...(\Lambda z_{n})^{j_{n}} \Lambda^{d-S},
$$
 where $S = \sum_{s=1}^{n} j_{s}$. Hence $\Lambda^{d} \times P$ is holomorphic on $\Omega$. Also, $\Lambda(x) \neq 0$.  Since $x$  is
in $h_{\Omega}(\gamma)$, we have
 \begin{equation}
   |(\Lambda^{d}P)(x)| \leq \max|\Lambda^{d}P| \ \ {\rm  over} \ \gamma.
 \end{equation}
We now argue as in [2], proof of Proposition 2.3:
It follows from (8) that
 $$
 |\Lambda(x)|^{d}\times|P(x)| \leq ( \max|\Lambda|)^{d} \times \max |P|,
$$
where the maxima are taken over $\gamma$. We now put
$$
 C_{x} = \frac{\max|\Lambda|}{|\Lambda(x)|}.
 $$
  Then
 \begin{equation}
  |P(x)| \leq{ C_{x}^{d} \max|P(x)|}
 \end{equation}
Since (9) holds for all $P$, we have that $x$ is in $\hat{\gamma}$. Thus $h_{\Omega}(\gamma)\subset \hat{\gamma}$.
So $h_{\Omega}(\gamma) = \hat{\gamma}$, and we are done.

\end{proof}

\section{The hull of a curve in a Stein manifold}

\begin{Theorem}
 Let $X$ be a Stein manifold, and let $\beta$ be a real-analytic closed curve in X. Then
$$
 h_{X}(\beta) = \beta \cup{V} 
$$
where $V$ is a 1-complex dimensional subvariety of 
$ X \setminus \beta$, $\beta$ and $V$ are disjoint, and $\beta$ is the boundary of $V$,

\end{Theorem}

\begin{proof}
Theorem 3.1 follows from the fact that it holds when $X = \C^{n}$ (\cite{W}), together with the following well-known properties of a Stein manifold $X$.

(a) $X$ admits a biholomorphic embedding  $\Phi$ into $\C^{N}$ for some $N$.

(b) Every holomorphic submanifold $Y$ of $\C^{N}$ is the zero set of some vector-valued entire function on $\C^{N}$, and

(c)Every holomorphic function on $Y$ admits a holomorphic extension to an entire funciton on $\C^{N}$.

\end{proof}

\noindent
{\bf Note 1: }\ If the Conjecture is true, then conditions (i) and (ii) in Theorem 1.1 are satisfied by $\gamma$. We see this as follows:

Put $V = \hat{\gamma} \setminus \gamma$.Assume that $V$ is a subvariety of $\bbp^{n} \setminus\gamma$, with boundary $\gamma$.
Then $\hat{\gamma} = V \cup {bd} V$, and so $\hat{\gamma}$ is closed in $\bbp^{n}$. So (i) holds.

Since $\gamma$ is real-analytic, $\gamma$ lies on some Riemann surface ``collar'' $S$, and $S$ fits together with $V$ to form a holomorphic subvariety $V^{*}$ of some open subset $O$ of $\bbp^{n}$, with $V^{*}$ a relatively closed subset of $O$. Then $V^{*}$ is a Stein subspace of $O$. Hence by a result of Siu, \cite{Si}, $V^{*}$ admits a Stein neighborhood $\Omega$
in $O$. Then

$\hat{\gamma}=V \cup{\gamma} \subset{V^{*}}\subset{\Omega}$. 
so (ii) holds.

\noindent
{\bf Note 2: }\  In Theorem 3.1, with $\beta$ assumed real-analytic, $\beta$ is the boundary of $V$ in the sense of ``manifold with boundary''.
If $\beta$ is merely assumed smooth, $\beta$ is the boundary of $V$ in a more general sense. (See, \cite{H-L{2}}, Th. 7.2).

\end{document}